\documentclass{amsart}
\usepackage{graphicx}
\usepackage{amssymb,amscd,amsthm,amsxtra}
\usepackage{latexsym}
\usepackage{epsfig}

\newtheorem{thm}{Theorem}[section]
\newtheorem{cor}[thm]{Corollary}
\newtheorem{lem}[thm]{Lemma}
\newtheorem{prop}[thm]{Proposition}
\theoremstyle{definition}
\newtheorem{defn}[thm]{Definition}
\theoremstyle{remark}
\newtheorem{rem}[thm]{Remark}
\numberwithin{equation}{section}

\newcommand{\be}{\begin{equation}}
\newcommand{\ee}{\end{equation}}

\newcommand{\R}{\mathbb R}

\newcommand{\eps}{\varepsilon}

\newcommand{\p}{\partial}

\newcommand{\comment}[1]{}

\begin{document}

\title[Nonlocal phase transitions]{Rigidity of minimizers in nonlocal phase transitions II} 
\author{O. Savin} 
\thanks{The author was partially supported by  N.S.F. Grant DMS-1500438}
\address{Department of Mathematics, Columbia University, New York, NY 10027}
\email{\tt  savin@math.columbia.edu}

\begin{abstract}
In this paper we extend the results of \cite{S3} to the borderline case $s = \frac 12$.  
We obtain the classification of global bounded solutions with asymptotically flat level sets for semilinear nonlocal equations of the type
$$\triangle^{\frac 12} u=W'(u) \quad \mbox{in} \quad \R^n,$$
where $W$ is a double well potential. 
\end{abstract}
\maketitle

\section{Introduction}

We continue the study initiated in \cite{S3} for the classification of global bounded solutions with asymptotically flat level sets for nonlocal semilinear equations of the type
$$\triangle^{s} u=W'(u) \quad \mbox{in} \quad \R^n,$$
where $W$ is a double well potential. 

The case $s \in (\frac 12, 1)$ was treated in \cite{S3} 
while $s\in (0, \frac 12)$ was considered by Dipierro, Serra and Valdinoci in \cite{DSV}. 
In this paper we obtain the classification of global minimizers with asymptotically flat level sets 
in the remaining borderline case $s = \frac 12$. 
All these works were motivated by the study of semilinear equations for the case of the classical Laplacian $s=1$, 
and their connection with the theory of minimal surfaces, see \cite{D,DKW,M,S1}. It turns out that when $s \in [\frac 12, 1)$, the rescaled level sets of $u$ still converge to a minimal surface while for $ s\in (0, \frac 12)$ they converge to an $s$-nonlocal minimal surface, see \cite{SV}. 

We consider the Ginzburg-Landau energy functional with nonlocal interactions corresponding to $\triangle^{1/2}$,
$$J(u, \Omega)=\frac14 \int_{\R^n \times \R^n \setminus ( \mathcal C \Omega \times \mathcal C \Omega)} 
 \frac{ (u(x)-u(y))^2}{|x-y|^{n+1}} \,dxdy+ \int_\Omega W(u) \, dx, $$ with $|u|\le 1$, and $W$ a double-well
potential with minima at $1$ and $-1$ satisfying
$$W \in C^2([-1,1]), \quad W(-1)=W(1)=0, \quad \mbox{$W>0$ on $(-1,1)$},$$
$$W'(-1)=W'(1)=0, \quad W''(-1)>0, \quad W''(1)>0.$$

Critical functions for the energy $J$ satisfy the Euler-Lagrange equation
$$\triangle ^{1/2} u=W'(u),$$
where $\triangle^{1/2} u$ is defined as
$$\triangle ^{1/2} u(x)= PV \int_{\R ^n} \frac{u(y)-u(x)}{|y-x|^{n+1}} \, dy . $$ 

Our main result provides the classification of minimizers with asymptotically flat level sets.

\begin{thm} \label{TM}
Let $u$ be a global minimizer of $J$ in $\mathbb{R}^n$. If the $0$ level set $\{u=0\}$ is asymptotically flat at $\infty$,
then $u$ is one-dimensional.
\end{thm}

The hypothesis that $\{u=0\}$ is asymptotically flat means that there exist sequences of positive numbers
$\theta_k$, $l_k$ and unit vectors $\xi_k$ with $l_k \to \infty$,
$\theta_k l_k^{-1}\to 0$ such that
$$\{u=0\} \cap B_{l_k} \subset \{ |x \cdot \xi_k| < \theta_k \}.$$

 By saying that $u$ is one-dimensional we understand that $u$ depends only on one direction $\xi$, i.e.  $u=g(x \cdot \xi)$.

As in \cite{S3}, we obtain several corollaries. We state two of them.

\begin{thm}\label{7min}
A global minimizer of $J$ is one-dimensional in dimension $n \le 7$.
\end{thm}

\begin{thm}{\label{8min}}
Let $u \in C^2(\mathbb{R}^n)$ be a solution of
\begin{equation}{\label{8min1}}
\triangle^{1/2} u = W'(u),
\end{equation}
such that
\begin{equation}{\label{8min2}}
|u| \le 1, \quad \partial_n u>0, \quad \lim_{x_n \to \pm
\infty}u(x',x_n) =\pm 1.
\end{equation}

Then $u$ is one-dimensional if $n \le 8$.

\end{thm}

Theorem \ref{7min} and Theorem \ref{8min} without the limit assumption in \eqref{8min2}, 
have been established by Cabre and Cinti \cite{CC} in dimension $n=3$. 
Recently, Figalli and Serra \cite{FS} obtained the same conclusion for all stable solutions in dimension $n=3$. 
Their result, combined with Theorem \ref{7min} above, implies the validity of 
Theorem \ref{8min}  without the limit assumption in \eqref{8min2}, in dimension $n=4$. 

We prove our result by making use of the extension property of $\triangle^{1/2}$. 
Let $U(x,y)$ be the harmonic extension of $u(x)$ in $\R^{n+1}_+$
$$\triangle U=0 \quad \mbox{in $\R^{n+1}_+$,} \quad U(x,0)=u(x), $$
then $$\triangle^{1/2}u(x)= c_nU_y(x,0),$$
with $c_n$ a dimensional constant.
Global minimizers of $J(u)$ in $\R^n$ with $|u| \le 1$ correspond to global minimizers of the extension energy $\mathcal J(U)$ with $|U| \le 1$ where
$$ \mathcal J(U):= \frac{c_{n}}{2} \int |\nabla U|^2 \, dx dy + \int W(u) dx.$$ 
After dividing by a constant and relabeling $W$ we may fix $c_{n}$ to be $1$. 

We obtain Theorem \ref{TM} from an improvement of flatness property for the level sets of minimizers of 
$\mathcal J$, see Proposition \ref{c1alpha}. 
We follow the main steps from \cite{S3,S2}, however some technical modifications are required. 
The main difference when $s=\frac 12$ is that at a point of $\{u=0\}$ 
which has a large ball of radius $R$ tangent from one side we can no longer estimate its curvatures in terms of $R^{-1}$. 
Instead we obtain an integral estimate (see Lemma \ref{l3}) which turns out to be sufficient for the key Harnack estimate of the level sets.

The paper is organized as follows. In Sections 2 we introduce some notation and then construct a family $G_R$ of axial supersolutions. In Sections 3 and 4 we provide viscosity properties for the mean curvature of the level set $\{u=0\}$. 
In Section 5 we obtain 
the Harnack inequality of the level sets and in Section 6 we prove our main result by compactness.

\section{Supersolution profiles}

We introduce the following notation.

\noindent
We denote points in $\R^n$ as $x=(x',x_n)$ with $x' \in \R^{n-1}$. The ball of center $z$ and radius $r$ is denoted by $B_r(z)$,
and $B_r:=B_r(0)$.
Points in the extension variables $\R^{n+1}_+$ are denoted by $X=(x,y)$ with $y>0$, and the ball of radius $r$ as $\mathcal B_r^+$
$$\mathcal B^+_r :=\{(x,y) \in \R^{n+1}_+| \quad |(x,y)| <r \} \quad  \subset \R^{n+1}.$$
Given a function $U(x,y)$ we denote by $u(x)$ its trace on $\{y=0\}$.

Let $\mathcal J$ be the energy 
$$ \mathcal J(U, \mathcal B_R^+):= \frac 12 \int_{\mathcal B^+_R} |\nabla U|^2 \, dx dy + \int_{B_r} W(u) dx,$$
and a critical function $U$ for $\mathcal J$ satisfies the Euler-Lagrange equation
\be\label{eq}
\triangle U=0, \quad \quad U_y=W'(u).
\ee

In \cite{PSV} it was established the existence and uniqueness up to translations 
of a global minimizer $G$ of $\mathcal J$ in 2D which is increasing in the first variable and which has limits $\pm 1$ at $\pm \infty$:

a) $G:\R^2_+ \to (-1,1)$ solves the equation \eqref{eq},

b) $G(t,y)$ is increasing in the $t$ variable and its trace $g(t):=G(t,0)$ satisfies
$$g(0)=0 , \quad \lim_{t \to \pm \infty} g(t)= \pm 1.$$
Moreover, $g$ and $g'$ have the following asymptotic behavior 
\be\label{gg}
1-|g| \sim \min\{ 1, |t|^{-1}\}, \quad \quad g' \sim \min \{ 1, |t|^{-2} \},
\ee
and 
$$
\mathcal J(G, B_R^+) = C^* \, \log R + O(1),
$$
for some constant $C^*$.

Constants that depend on $n$, $W$, $G$ are called universal constants, and we denote them by $C$, $c$. 
In the course of the proofs the values of $C$, $c$ may change from line to line when there is no possibility of confusion. If the constants depend on other parameters, say $\theta$, $\rho$, then we denote them by $C(\theta, \rho)$ etc. 

For simplicity of notation we assume that 
\be\label{W1}
\mbox{$W$ is uniformly convex outside the interval $[g(-1),g(1)]$.}
\ee

Since $\triangle G_t=0$ and $G_t \ge 0$, we easily conclude that
\be\label{NG}
|G_y| \le C \frac{1}{1+r}, \quad \quad C'\, \frac {1+y}{1+r^2} \ge G_t \ge c'  \, \frac {1+y}{1+r^2},
\ee
where $r$ denotes the distance to the origin in the $(t,y)$-plane. We also obtain
$$|G| \le 1- c r^{-1}, \quad \quad \forall \, r \ge 1.$$
Define in $\mathcal B_{R/8}^+$,

\be\label{hr}
H_R(t,y):=G(t,y) + \frac {C_0}{R} \left((y+C_1)\log R- y \log(y+1) + \frac{1}{R} (t^2-y^2)    \right),
\ee
for some $C_0$, $C_1$ universal, large to be made precise later.
 
Let $\bar H_R$ denote the truncation of $H_R$ at level 1,
$$ \bar H_R = \min\{ H_R,1\}.$$

Since $$(y+C_1)\log R- y \log(y+1) - \frac 1 R y^2  $$
is strictly increasing in the interval $[0, R/8]$ we conclude that
$$
\mbox{$ H_R  \ge G + 2 > 1$ if either $y> C R / \log R$ or $|t|>R/16$.}
$$ 
Hence $\bar H_R=1$ outside $ \mathcal B_{R/16}^+$, and we extend $\bar H_R =1$ outside this ball in the whole $\R^2_+$. 
Finally we define $G_R$ in $\R^2_+$ as 
$$G_R(t,y): = \inf_{l \ge 0} \bar H_R(t+l,y).$$

Next we collect some key properties of the function $G_R$.
\begin{lem}[Supersolution profile]\label{gr} Then for all large $R$ we have

1) $G_R=1$ outside $\mathcal B_{R/8}^+ \cup \left ( (-\infty,0] \times [0,R/8] \right )$, 

\

2) $G_R(t,y)$ is nondecreasing in $t$, and $\p_t G_R=0$ outside $\mathcal B_{R/8}^+$,

\

3) $G_R = H_R$ in $ \mathcal B_{R^{1/3}}^+ $ and 
$$ |G_R-G| \le C \,  \frac {\log R}{R}  \quad \quad \mbox{in} \quad  \mathcal B_{4}^+,$$

\

4)
$$\triangle G_R+ \frac{2(n-1)}{R}\,   \partial_t G_R \le 0, $$
and on $y=0$: 
$$\partial_y G_R < W'(G_R) + \chi_{[-1,1]} \frac{C \log R}{R} .$$

\end{lem}

The inequalities in 4) are understood in the viscosity sense.

Notice that by \eqref{NG}, property 3) implies that $$G_R(t,y) \le G \left(t+ C''\,\frac{\log R}{R},y \right) \quad \mbox{ in}  \quad \mathcal B_4^+.$$

\begin{proof}
Properties 1) and 2) follow from the definition of $G_R$ since  $\bar H_R=1$ outside $\mathcal B_{R/8}^+$.
We compute 
\be\label{21}
\partial_t H_R= G_t + \frac{2 t C_0}{R^2},
\ee
and use \eqref{NG} to conclude that $\partial _t H_R >0$ if $t \ge - R^{1/2}$. In $\mathcal B_{R^{1/3}}^+$ we have 
$$|H_R| \le |G| + \frac{C_0}{R}(R^{1/2}+R^{-1/3}) \le 1- c R^{-1/3} + C R^{-1/2} < 1,$$
hence 
$$G_R=H_R \quad \mbox{in} \quad \mathcal B_{R^{1/3}}^+.$$

For property 4) we use that $G_R$ is the infimum over a family of left translations of $H_R$, hence it suffices to show the inequalities for $H_R$ in the region where $H_R<1$ and $\partial_t H_R>0$. This means that we can restrict to the region where $ 1+y \le C R/ \log R,$ and $ |t| \le C R^{2/3}$.
From the definition \eqref{hr} of $H_R$ we have
$$\triangle H_R = - \frac{C_0}{R}\left (\frac{1}{(1+y)^2} + \frac{1}{1+y}\right) \le -\frac{C_0}{R} \frac{1}{1+y},$$
and, by \eqref{21}
$$\partial _t H_R \le C' \frac{1+y}{1+r^2} + C_0 \frac {r}{R^2}.$$
Since $1+y \le C R/ \log R$ we easily obtain the first inequality in 4) by choosing $C_0$ large depending on $C'$.
On $y=0$ we have 
$$\partial_y H_R = \partial_y G + C_0 \frac{\log R}{R} = W'(G)  + C_0 \frac{\log R}{R} . $$
$$ H_R=G + C_0C_1 \frac{\log R}{R} + C_0 \frac{t^2}{R^2}.$$
Then, by \eqref{W1}, $W''(g)>c$ outside the interval $[-1, 1]$, and we find that  
$$ W'(H_R) \ge W'(G) + c C_1 C_0 \frac{\log R}{R} > \partial_y H_R,$$
which easily gives the desired conclusion.

\end{proof}

\section{Estimates for $\{u=0\}$}

In this section we derive properties of the level sets of solutions to
\begin{equation}{\label{2min1}}
\triangle U = 0, \quad \quad \p_y U=W'(U),
\end{equation}  
which are defined in large domains.

In the next lemma we use the functions $G_R$ constructed in the previous section and find axial approximations to \eqref{2min1}.

\begin{lem}[Axial approximations]\label{l1}
Let $G_R: \R^2_+ \to (-1,1]$ be the function constructed in Lemma \ref{gr}. Then its axial rotation in $\R^{n+1}$
$$\Phi_R (x,y):=G_R(|x|-R,y)$$
satisfies 

1) $\Phi_R=1$ outside $\mathcal B_{2 R}^+$, and $\Phi_R$ is constant in $\mathcal B_{ R/2}^+$. 

\

2) 
$$\triangle \Phi_R \le 0 \quad \mbox{ in} \quad \R^{n+1}_+,$$ and $$\p_y \Phi_R < W'(\Phi_R) \quad \mbox{when} \quad |x|-R \notin  [-1,1] ,$$

\

3) In the annular region $|(|x|-R,y)| \le R ^\frac 13,$ we have 
$$|\triangle \Phi_R| \le C \frac 1 R, \quad \quad |\partial_y \Phi_R - W'(\Phi_R)| \le C \frac {\log R}{R}. $$

\end{lem}

Let $\phi_R(x)=\Phi_R(x,0)$ denote the trace of $\Phi_R$ on $\{y=0\}$. 
Notice that $\phi_R$ is radially increasing, and $\{\phi_R=0\}$ is a sphere which is in a $C \log R /R$-neighborhood of the sphere of radius $R$.

\begin{proof}

We have
$$\triangle  \Phi_R(x,y)= \triangle  G_R(s,y)  + \frac{n-1}{R+s}\, \partial_s \, G_R(s,y), \quad \quad s=|x|-R,$$
$$\p_y \Phi_R(x,0)=\p_y G_R(s,0).$$
The conclusion follows from Lemma \ref{gr} since $\partial_s G_R=0$ 
when $|s| \ge R/8$ and $R+s>R/2$ when $|s| < R/8$.

\end{proof}

\begin{defn}\label{d1}
 We denote by $\Phi_{R,z}$ the translation of $\Phi_R$ by $z$ i.e.
$$\Phi_{R,z}(x,y):= \Phi_R(x-z,y)=G_R(|x-z|-R,y).$$

\end{defn}

{\it Sliding the graph of $\Phi_R$:}

Assume that $u$ is less than $\phi_{R,x_0}$ in $B_{2R}(x_0)$. 
By the maximum principle we obtain that $U< \Phi_{R,z}$ with $z=x_0$ 
in $\mathcal B_{2R}(x_0,0)$ (and therefore globally.)
We translate the function $\Phi_R$ above by moving continuously the center $z$, 
and let's assume that it touches $U$ by above, say for simplicity when $z=0$, 
i.e. the strict inequality becomes equality for some contact point $(x^*,y^*)$. 
From Lemma \ref{l1} we know that $\Phi_R$ is a strict supersolution away from $\{y=0\}$, 
and moreover the contact point must satisfy $y^*=0$, $|x^*| -R \in [-1,1]$, 
that is it belongs to the annular region $B_{R+1}\setminus B_{R-1}$ in the $n$-dimensional subspace $\{y=0\}$.

\begin{lem}[Estimates near a contact point] \label{l2}
Assume that the graph of $\Phi_R$ touches by above the graph of $U$ at a point $(x^*,0, u(x^*))$ with $x^* \in B_{R+1} \setminus B_{R-1}$. 

Then in $B_2(x^*,0)$
the level set $\{u=0\}$ stays in a $C \log R/R$ neighborhood of the sphere $\p B_R=\{|x|=R \}$, and 
$$\|u-\phi_R\|_{C^{1,1}(B_2(x^*))} \le C \, \frac{\log R}{R}.$$

\end{lem}

\begin{proof}

Assume for simplicity that $x^*$ is on the positive $x_n$ axis, thus $|x^*-Re_n| \le 1$. By Lemma \ref{l1} we have
$$U \le \Phi_R \le G \left (x_n-R + C \, \frac{\log R}{R},y \right)=:V \quad \quad \mbox{in} \quad \mathcal B_{3}(R e_n).$$
Both $U$ and $V$ solve the same equation \eqref{2min1}, and $$(V-U)(x^*,0) \le C\, \frac {\log R}{ R}.$$ Since $V-U \ge 0$ satisfies
$$\triangle (V-U)=0, \quad \quad \p_y (V-U)= b(x)(V-U), $$
$$ b(x):= \int_0^1 W''(tu(x) + (1-t)v(x)) dt,$$
we obtain $$|V-U| \le C \, \frac {\log R} {R} \quad \quad \mbox{in} \quad \mathcal B_{5/2}(R e_n),$$
from the Harnack inequality with Neumann boundary condition. Moreover since $b$ has bounded $C^{1,\alpha}$ norm, we obtain that $U-V \in C_x^{2,\alpha}$ for some $\alpha>0$, and 
$$\|U-V\|_{C^{1,1}(\mathcal B_2(Re_n))} \le C\, \frac{\log R}{ R},$$
by local Schauder estimates. This easily implies the lemma.

\end{proof}

\begin{rem}\label{r2}
If instead of $\Phi_R$ being tangent by above to $U$, we only assume $$ \Phi_R \ge U \quad \mbox{ and}  \quad (\phi_R-u)(x^*) = :a \le c $$ at some $x^* \in B_{R+1} \setminus B_{R-1}$ then the Harnack inequality above gives
$$ c a - C\,  \frac {\log R} {R} \le \Phi_R-U   \le C \left (a + \frac{\log R}{R} \right) \quad \quad \mbox{in} \quad \mathcal B_2(x^*,0),$$
for some $C$ large, universal. 
\end{rem}

\begin{lem}\label{l3}
Assume that 

a) $B_R(-Re_n) \subset \{u<0\}$ is tangent to $\{u=0\}$ at $0$. 

b) there is $x_0 \in B_{R/2}(-Re_n)$ such that $u(x_0) \le -1+c$ for some $c>0$ small. 

Denote by $D$ the set 
$$D:=\{u<0\} \setminus B_R(-Re_n).$$  Then
\be\label{225}
\int_{B_1^c} \frac {\chi_D(x)}{|x|^{n+1}} \, \, dx \le C \frac{\log R}{R},\ee
and $$\{u=0\} \cap B_{R^\sigma} \subset \{|x_n| \le R^{-3/4}\},$$
for some $\sigma>0$ small, universal.

\end{lem}

We remark that in \eqref{225} we can integrate over whole $\R^n$ instead of $B_1^c$ since, by Lemma \ref{l2}, the curvatures of $\p D$ in $B_1$ are bounded by $C \log R/R$.

\begin{proof}
First we claim that 
\be\label{22}
U \le \Phi_{R/2, t_0 e_n}, \quad \quad \mbox{with} \quad t_0=-\frac R 2 - K \frac {\log R} {R},
\ee
for some $K$ large universal.

Let's assume first that $\Phi_{R/2,t e_n} \ge U$ when $t=-R$. We want to show that this inequality remains valid as we increase $t$ from $-R$ till $t_0$. By Lemma \ref{l1}, the first contact point between the graphs of $U$ and $\Phi_{R, t e_n}$ can occur only on $y=0$ and, by Lemma \ref{l2} near this contact point the $\{u=0\}$ and $|x+te_n|=R/2$ must be at most $C \log R/R$ apart. This is not possible if $K$ is chosen sufficiently large. 

To prove that $\Phi_{R/2,-R e_n} \ge U$, one can argue similarly by using hypothesis b) and looking at the continuous family $\Phi_{r, -Re_n}$ and then increase 
$r$ from $C$ to $R/2$. This proves the claim \eqref{22}.

We write $\Phi=\Phi_{R/2,t_0e_n}$ for simplicity of notation, and by $\phi$ the trace of $\Phi$ on $y=0$. We have $\Phi \ge U$, and therefore $\phi \ge u$, and
$$\triangle^{1/2}\phi \, (0) \le W'(\phi(0)) + C \frac{\log R}{R}.$$ 
Using that $$|\phi(0)| \le C(K) \log R/R, \quad \quad u(0)=0,$$ together with the equation for $ \triangle^{1/2} u$ at $0$ we obtain that
$$\triangle^{1/2}(\phi-u)(0) \le C(K) \frac{\log R}{R}.$$
Since $(\phi-u)(0)$ and $\|\phi-u\|_{C^{1,1}(B_1)}$ (by Lemma \ref{l2}) are bounded by $C \log R/R$, we use the integral representation for $\triangle ^{1/2}$ and obtain
\be\label{23}
\int_{B_1^c} \frac {\phi-u}{|x|^{n+1}}dx \le C \frac{\log R}{R}.
\ee
Next we show that we can replace $\phi-u$ in the integral above by $\chi_{\tilde D}$ where $$\tilde D:=\{u<0\} \setminus \{\phi<0\} \supset D.$$ For this it suffices to show that for any unit ball $B_1(z)$ with center $z \in \tilde D$ we have
$$c_1 \int_{B_1(z)} \chi_{\tilde D} dx \le  \int_{B_1(z)} (\phi-u) + C_1 \frac{\log R}{R} dx,$$
for some $c_1>0$ small, and $C_1$ large universal. 

Indeed, let $a=(\phi-u)(z)$. If $a>c$ then, by the Lipschitz continuity of $\phi-u$, the right hand side above is bounded below  by a universal constant and the inequality is obvious. If $a<c$ then we use Remark \ref{r2} and conclude that
$$|\tilde D \cap B_1(z) | \le C \left(a + \frac{\log R}{R} \right) ,$$
and $$\phi-u \ge ca - C \frac{\log R}{R} \quad \mbox{in} \quad B_1(z),$$
which gives the desired inequality by choosing $C_1$ sufficiently large.

Next we show that 
\be\label{24}
\varphi-u \le R^{-4/5} \quad \mbox{ in} \quad B_{R^\sigma},
\ee 
for some small $\sigma>0$. Assume by contradiction that
$$(\varphi-u)(z) > R^{-4/5} \quad \quad \mbox{for some} \quad z \in B_{R^\sigma}.$$
Let $V:=\Phi-U \ge 0$. We have $V(z,0) > R^{-3/4}$, and by part 3) of Lemma \ref{l1},  
$$  |\triangle V| \le \frac{C}{R}, \quad |\partial_y V| \le C V + C\,  \frac{\log R}{R} \quad \quad \mbox{in} \quad \mathcal B_2(z) .$$
By Harnack inequality we obtain
$$V \ge c R^{-4/5} \quad \mbox{in} \quad \mathcal B_{3/2}(z).$$
This means that the left hand side in \eqref{23} is greater than $c R^{-4/5} R^{-(n+1)\sigma}$ and we reach a contradiction if we choose $\sigma$ small depending only on $n$.
Hence the claim \eqref{24} is proved. This implies that in $B_{R^\sigma}$, the set $\{u=0\}$ is in a $CR^{-4/5}$ neighborhood of the $0$ level set of $\phi$ which gives the desired conclusion.

\end{proof}

\begin{rem}\label{r3}
In the proof above we obtain 
\be\label{25}
|\Phi-U| \le R^{-3/4}  \quad \mbox{in} \quad \mathcal B_{R^\sigma}.
\ee
Indeed, in $\mathcal B_{R^\sigma}$ we have
$$ V \ge 0, \quad   |\triangle V| \le \frac{C}{R}, \quad |\partial_y V| \le C R^{-4/5},$$
where in the last inequality we used that on $y=0$, $ V \le R^{-4/5}$ by \eqref{24}.
Now \eqref{25} follows from Harnack inequality provided that $\sigma$ is sufficiently small.
\end{rem}

As a consequence of Lemma \ref{l3} we obtain
\begin{cor}\label{c1}
Assume that $B_R(-Re_n) \subset \{u<0\}$ is tangent to $\{u=0\}$ at $0$. Then in the cylinder $\{|x'| \le R^\frac \sigma 6\}$ the set $\{u=0\}$ cannot lie above the surface
$$x_n = \frac 1 R \left( \Lambda (x \cdot e')^2 -|x'|^2 \right),$$
where $e'$ is a unit direction with $e' \cdot e_n=0$ and $\Lambda$ is a large universal constant. 
\end{cor}

\begin{proof}
Indeed, otherwise the integral in \eqref{225} is greater than
$$\int_1^{R^{\sigma/6}} c \,  \frac{\Lambda}{R} \,  \frac{r^n}{r^{n+1}} \, \, dr \ge c(\sigma) \Lambda \, \, \frac{\log R}{R},$$
and we reach a contradiction if $\Lambda$ is chosen sufficiently large.

\end{proof}

\section{A mean curvature estimate for $\{u=0\}$}

In this section we refine some of the results of last section and we estimate the mean curvature of a surface that touches $\{u=0\}$ by below at $0$, in a neighborhood of size $l \ge R^{1/3}$.
\begin{prop}\label{p1}
Fix $\delta>0$ small and let $R$, $l$ be large with $l \in [R^{1/3}, \delta^3 R]$, and let $\theta$ denote
$$\theta:=l^2 R^{-1}.$$
Assume that in the ball $B_l$ the surface 
$$\Gamma:=\left \{x_n= \sum_1^{n-1} \frac{a_i}{2} x_i^2 + b' \cdot x'  + b_0\right \},$$
with
$$|a_i| \le \delta R^{-1}, \quad |b'| \le \delta l R^{-1}, \quad |b_0| \le \theta, $$
is tangent to $\{u=0\}$ at $b_0 e_n$. 

Assume further that $\{u<0\}$ contains the two balls $B_{R}(-t_0 e_n)$ and $B_{R_m}(-t_m e_n)$ of radii $R$ and $R_m$ and passing through $-\theta e_n$ and respectively $-\theta_me_n$ with
$$t_0 =\theta + R, \quad \quad t_m= \theta_m + R_m, \quad \quad R_m:=2^{\frac 12 m}R, \quad  \theta_m:= 2^{\frac 32 m}\theta.$$   
Then
$$\sum_1^{n-1} a_i \le \delta^4 R^{-1},$$
if $m=m(\delta)$ is chosen sufficiently large depending only on $\delta$ and the universal constants.
\end{prop}

\begin{proof}
First we claim that at each point $x_0 \in \Gamma \cap B_{l/3}$ we have a tangent ball of radius $\frac {1}{16} R$ by below which is included in the set $\{u<0\}$. 

Indeed, the bounds on $|a_i|$, $|b_i|$ imply that at $x_0$, $\Gamma$ has a quadratic polynomial $$x_n=- \frac {8}{R} |x'-z_0'|^2 + c_{z_0}$$ tangent by below, with $$|x_0'-z_0'| \le C \delta l,  \quad \quad c_{z_0} \le 2 \theta. $$ It is straightforward to check that the quadratic surface above lies inside the ball $B_{R}(t_0 e_n)$ in the region $B_R \setminus B_l$, and our claim easily follows.

As in the proof of Lemma \ref{l2}, $B_{R_m}(t_m e_n) \subset \{u<0\}$ gives the bound
$$U \le \Phi_{R_m /2, t e_n},$$
as long as the ball $B_{R_m/2}(te_n)$ lies inside the ball
$$\tilde B:=B_{\tilde R}(t_m e_n), \quad \quad \tilde R:= R_m- C \frac{\log  R_m}{R_m},$$  
for some $C$ large, universal. This gives the bound
\be\label{26}
U \le G_{R_m/2} \left (d_m + C \log  R_m / R_m,y \right),
\ee
where $d_m$ denotes the signed distance to the sphere $ \p B_{R_m}(-t_m e_n)$, with $d_m>0$ outside the ball.

Similarly, we use that at each point in $\Gamma \cap B_{l/3}$ the tangent ball of radius $\frac {1}{16}R$ by below is included in $\{u<0\}$ and we obtain
\be\label{27}
U \le G_{R/32}(d_\Gamma + C \log R /R) \quad \mbox{in} \quad \mathcal B_{l/4},
\ee
where $d_\Gamma$ represents the signed distance to the the surface $\Gamma$.

We assume by contradiction that the conclusion is not satisfied i.e.
$$\sum a_i > \delta^4 R^{-1}.$$ Since on $\Gamma \cap B_l$ the slope of $\Gamma$ (viewed as a graph in the $e_n$ direction) is bounded by $C(n) \delta l R^{-1} \le \delta^2$ and $|a_i|\le \delta R^{-1}$ we obtain 
$$H_\Gamma \ge \sum a_i - C \delta^4 \max |a_i| \ge \frac 12 \delta^4 R^{-1},$$
where $H_{\Gamma}$ represents the mean curvature of $\Gamma$. Moreover, the curvatures of $\Gamma$ are bounded by $ 2 \delta R^{-1}$, which easily gives that in $B_l$ all parallel surfaces to $\Gamma$ satisfy a similar mean curvature bound:
\be\label{28}
H_\Gamma(x) \ge  \sum a_i- C l (\delta R^{-1})^2 \ge \frac 12 \delta^4 R^{-1}, \quad \quad \forall \, x \in B_l,
\ee
where we have used the hypothesis $l R^{-1} \le \delta^3$.
Here $H_\Gamma(x)$ denotes the mean curvature of the parallel surface to $\Gamma$ passing through $x$.
 
Next we use \eqref{28} to construct a supersolution with $0$ level set sufficiently close to $\Gamma$. Then we make use of \eqref{26}, \eqref{27} and reach a contradiction by showing that this supersolution touches $U$ by above at an interior point.  

For the construction of the supersolution we first introduce a 2D profile in the $(t,y)$ variables which is a perturbation of $G$. It is similar to the profile $H_R$ defined in \eqref{hr}. Precisely we define $H^*$ in $\R^2_+$ as  
\be\label{H*}
H^*(t,y):=G + \frac{c(\delta)}{R} h(t,y)
\ee
with
\be\label{h}
h(t,y):= c_1 \,  \varphi (2r) \, \, y \log r + \varphi(r)\, \frac{t^2-y^2}{r},
\ee
where $r=|(t,y)|$ is the distance from $(t,y)$ to the origin, and $\varphi$ is a cutoff function with $\varphi=0$ in $[0,1]$ and $\varphi=1$ in $[2,\infty)$. The constant $c_1$ is small, universal, and the constant $c(\delta)>0$ depends also on $\delta$ will be made precise below. Outside $\mathcal B_4^+$, the function $h$ has the property that  

a) $\triangle h$ is homogenous of degree $-1$ and 

b) on $y=0$, $h=|t|$ and $h_y=c_1 \log |t|$. 

The following properties hold provided that $c_1$ is sufficiently small:

1) $h$ is superharmonic in an angular region near the $t$ axis
$$\triangle h < 0 \quad \mbox{in the region} \quad \{y < |t|/2\} \setminus \mathcal B^+_4,$$

2) Outside this region we have the bounds (see \eqref{NG})
$$ \triangle h \le C \min\{1, r^{-1}\},  \quad \quad \partial _t H^* \ge c \min\{1,r^{-1}\},$$

3) on $y=0$ we have $h(t,0) \ge 0$ and 
$$ \mbox{$h_y =h=0$ in $ [-1,1]$ and $h_y \le c h$ outside $[-1,1]$,} $$
for some $c>0$ universal, smaller than the minimum of $W''$ outside the interval $[g(-1),g(1)]$ (see \eqref{W1}).

These properties imply that in the region where $\p_t H^* \ge 0$ we have
$$ \triangle H^* - \frac {\delta^4}{8} \frac 1 R\, \p_t H^* \le 0,$$
provided that $c(\delta)$ is chosen sufficiently small, and
$$ \p_y H^* \le W'(H^*) \quad \mbox{on $y=0$}.$$ 
 Next we modify the 2D profile $H^*$ by cutting at level 1 and making it increasing in the $t$ variable. We define $G^*$ as the infimum over left translations in a similar fashion as we did for $H_R$. Precisely, we define
 $$\bar H^* := \min\{H^*,1\}, \quad \quad G^*:= \inf_{l \ge 0} \bar H^*(t+l,y),$$
  and then $\partial _t G^* \ge 0$ by construction. Moreover, $G^*$ satisfies the inequalities above (in the viscosity sense):
   \be\label{285}
 \triangle G^* - \frac {\delta^4}{8} \frac 1 R \, \p_t G^* \le 0, 
 \ee
 and
 \be\label{286}
 \p_y G^* \le W'(G^*) \quad \mbox{on $y=0$.}
 \ee 
 
 In the next lemma we compare the profiles $G^*$ with appropriate translations of $G_{R/32}$ respectively $G_{R_m/2}$.
 
 \begin{lem}\label{ggr} We have the following inequalities:
 
 a) on $y=0$
 \begin{align}
 G^*(t,0) & \ge G_{R/32}(t-R^{-1/2},0) \\
 G^*(t,0)  & \ge G_{R/32} (t + R^{-1/2},0), \quad \mbox{if $|t|> R^{1/4}$}.
 \end{align}
 
 b) in $\R^2_+$ we have
 \begin{align}
\label{ggr3} G^*(t,y) & \ge G_{R/32}(t + K \log R/R,y) - C(K) \frac{log R}{R} (y+1),\\
\label{ggr4} G^*(t,y) & \ge G_{R_m/2}(t + 2 \theta_m,y) + c_1(\delta) \frac{log R}{R} y, \quad \mbox{if } \quad y \ge l ( \log R)^{- \frac 13}.
 \end{align}
 provided that $k=k(\delta)$ is chosen sufficiently large.
 \end{lem}
 
 \begin{proof}
 It suffices to show the inequalities for $H^*$ and $H_R$ in the regions where $\{\p _t H^*>0\} \cap \{ H^*<1\}$ and then the desired results for $G^*$ and $G_R$ follow by taking the infimum over left translations. 
First we check that
\be\label{287}
 \{\p _t H^*>0\} \cap \{ H^*<1\} \subset \{ r \le  R ( \log R)^{-1/3}\}.
\ee
We notice from \eqref{H*} that
\be\label{288}
h(t,y) \ge c (y \log r + r) \quad \mbox{outside $\mathcal B_C^+$.}
\ee
This means that if $y > CR/ \log R$ then $H^* >1$. In the two regions where $y < C R/\log R$ and $r> R (\log R)^{-1/3}$ we have

a) either $t> r/2$ and then we easily obtain $H^*>1$ by using (see \eqref{NG})  
$$G \ge 1 - C \frac{1+y}{r}$$

b) or $t < -r/2$ and we obtain $H^*_t <0$ by using
$$G_t \le C ( 1+y) r^{-2}, \quad \mbox{and} \quad h_t \le -c,$$
and \eqref{287} is proved.

To prove a) we have (see \eqref{h}, \eqref{hr}) 
 $$H^*(t,0) \ge g(t) +\left (1-\chi_{[-2,2]} \right )\frac{c(\delta)|t|}{R},$$
 $$H_{R/32}(t\pm R^{-1/2},0) \le g(t \pm R^{-1/2}) + C \frac{\log R}{R} + C \left(\frac t R \right)^2.$$
 The two inequalities follow easily since $|t|/R = o(1)$ by \eqref{287}, and by \eqref{gg} we have
 $$g(t-R^{-1/2},0) \le g(t) - c  \frac{R^{-1/2}}{1+t^2}, \quad \mbox{and} \quad g(t+R^{-1/2},0) \le g(t) + C  \frac{R^{-1/2}}{1+t^2}  .$$

 For part b) we estimate translations of $H_R$ as
$$H_R(t + \sigma,y) \le G(t + \sigma,y) + C \frac{\log R}{R} (y+1) + C\frac{t^2+\sigma^2}{R^2}  $$
and using that $G_t \le C/(y+1)$ we have
 $$H_R(t + \sigma,y) \le G(t,y) +  C \frac{\sigma}{y+1} + C \frac{\log R}{R} (y+1) + C \left(\frac t R \right)^2.$$

The third inequality is easily verified by taking $\sigma = K \log R / R $, $C(K)$ sufficiently large and then using \eqref{288} to estimate $H^*$.

Finally, for \eqref{ggr4} we take $\sigma=2 \theta_m=2^{\frac 32 m+1} \theta$ and
 replace $R$ with $\frac 12 R_m=2^{\frac 12 m-1}R.$ in the inequality above. 
 
We restrict to the region $y \ge  l (\log R)^{-1/3} \ge R^{1/4}$, thus we have $r \ge R^{1/4}$. We first choose $m$ large such that
$$\frac c 3 \frac{\log r}{ R} y \ge  C \frac{\log R_m}{R_m} (y+1),$$
 and then
$$\frac c 3\frac {\log r}{ R} y  \ge C \frac{2 \theta_m}{y+1},$$
for all large $R$'s, and the lemma is proved.

 \end{proof}
 
In the ball $\mathcal B_{l/4}$ we define the function 
$$\Psi:=G^*(d_{\tilde \Gamma},y)$$
where $d_{\tilde \Gamma}$ is the signed distance to the surface
$$\tilde \Gamma = \left \{x_n= \sum_1^{n-1} \frac{a_i}{2} x_i^2  - \frac{\delta^4}{4nR} |x'|^2+ b' \cdot x' \right\}. $$

From the properties of $G^*$ we find that $\Psi$ is a supersolution which is increasing in the $e_n$ direction. Indeed,
at a point $(x,y) \in \mathcal B_{l/4}$ we have (as in \eqref{28}) $$H_{ \tilde \Gamma}(x) > \frac 18 \delta^4 R^{-1},$$
and we compute (see \eqref{285},\eqref{286})
$$\triangle \Psi (x,y) = \triangle G^*(s,y) - H_{ \tilde \Gamma}(x) \partial_s G^*(s,y) <0. $$
where $s=d_{\tilde \Gamma}$. Also, on $\{y=0\}$ 
$$\partial_y \Psi =\partial_y G^*(s,0) \le W'(G^*) = W'(\Psi).$$

We claim that on $\{y=0\}$ we have
\be \label{psu}
\Psi > U \quad \mbox{outside} \quad  B_{l/8}.
\ee
To prove this, in view of \eqref{27}, it suffices to show that 
$$ G^*(d_{\tilde \Gamma},0) > G_{R/32} (d_1,0), \quad \quad d_1:=d_{\Gamma} + C \log R/R.$$
Indeed in $B_{l/4} \setminus B_{l/8}$ we either have 

a) $d_{\tilde \Gamma} > d_\Gamma + c l^2 R^{-1} > d_1 + R^{-1/2}$ or,

 b) $|d_{\tilde \Gamma}| \ge l/16 > R^{1/4}$ and $d_{\tilde \Gamma} \ge d_\Gamma > d_1 - R^{-1/2}$. 

The claim follows then from part a) of Lemma \ref{ggr} above.

Moreover, in $ \mathcal B_{l/4}$ we have $d_{\tilde \Gamma} \ge d_\Gamma $ hence by \eqref{ggr3}
$$\Psi > U - C \frac{\log R}{R} (y+1).$$
Since $$d_{\tilde \Gamma} + 2 \theta_m \ge d_m + C \log R_m /R_m,$$ (recall \eqref{26} for the definition of $d_k$) we find by \eqref{ggr4}, \eqref{26}, that
$$\Psi > U + c \frac{\log R}{ R} y \quad \mbox{if} \quad y > l \, \, (\log R)^{-\frac 13}  =: l \eps_R.$$

Thus on $\p \mathcal B_{l/4}$ we have 
$$\Psi > U - C \gamma \quad \mbox{if $y< l (2\eps_R)$, and} \quad \Psi> U + c \gamma \quad \quad \mbox{otherwise},$$
where $$\gamma:= \frac{ \log R}{R} (l \eps_R).$$

Next we translate the graph of $\Psi$ in the $-e_n$ direction till, on $y=0$ it becomes tangent by above to the graph of $U$. 
Indeed, since $\Psi(0)=U(0)$ and $\Psi >U$ outside $B_{l/8} \times \{0\}$, we can translate $\Psi$ so that $\Psi_0(X):=\Psi(X + t_1 e_n)$, for some $t_1 \ge 0$, becomes tangent by above to $U$ on $y=0$ at some point $(x^*,0) \in B_{l/8}$.  

Now we see that $V:= \Psi_0 - U \ge \Psi-U$ satisfies in $B_{l/4}$:
$$\triangle V \le 0, \quad V \ge 0 \quad \mbox{on} \quad \{y=0\}, \quad V (x^*,0)=0,$$
and on $\p \mathcal B_{l/4}$,
$$V \ge - C \gamma \quad \mbox{if $y< l (2\eps_R)$, and} \quad V \ge c \gamma \quad \quad \mbox{otherwise}.$$
Since $\eps_R$ can be taken arbitrarily small we find
$V \ge 0 $ in $\mathcal B_{3l/8}$ and therefore we also obtain $V_y(x^*,0)>0$. This means 
$$ U_y < \p_y \Psi_0 \le W'(\Psi_0)=W'(U) \quad \quad \mbox{ at $(x^*,0)$},$$ 
and we reached a contradiction. 

\end{proof}

\section{Harnack inequality}

In this section we prove a Harnack inequality property for flat level sets, see Proposition
\ref{TH} below. We will make use of Lemma \ref{l3} and Corollary \ref{c1} together with a standard $\Gamma$-convergence result for minimizers, see Lemma \ref{l1conv}.

\

{\it Notation:} We denote by $\mathcal C(l, \theta)$ the cylinder
$$\mathcal C(l,\theta):= \{|x'| \le l, \quad |x_n| \le \theta\}.$$

\begin{prop}[Harnack inequality for minimizers]{\label{TH}}

Let $U$ be a minimizer of $\mathcal J$ and assume that
$$ \{u=0\} \cap \mathcal C(l,l) \subset \mathcal C(l,\theta),$$
and that the balls of radius $C'l^2 \theta^{-1}$ (with $C'$ universal) which are tangent to 
$\mathcal C(l,\theta)$ at $\pm \theta e_n$ by below and above are included in $\{u<0\}$ respectively $\{u>0\}$.

Given $\theta_0>0$, there exist $\varepsilon_0(\theta_0) >0$ depending on $\theta_0$, such that if
\be\label{e4}
\theta l^{-1} \le \varepsilon_0(\theta_0), \quad \theta_0 \le \theta,\ee
then
$$ \{u=0\} \cap \mathcal C(\frac l 4, \frac l 4)$$ is either included in $\{x_n \le (1-\omega_0)\theta \}$ or in $\{ x_n \ge -(1-\omega_0) \theta\}$, with  $\omega_0>0$ small universal.
\end{prop}

After a translation in the $e_n$ direction, the conclusion can be stated as 
$$\{u=0\} \cap \mathcal C(\bar l, \bar l)\subset \mathcal C (\bar l, \bar \theta) \quad 
\mbox{ with} \quad \bar l:=\frac l 4, \quad \bar \theta:= (1- \frac {\omega_0}{2}) \theta.$$
 We remark that if \eqref{e4} is satisfied again for $\bar \theta$, $\bar l$, then we can apply Proposition \ref{TH} again since the hypothesis that the tangent ball of radius $C' {\bar  l}^2 {\bar \theta}^{-1}$ tangent by below to $\mathcal C(\bar l, \bar \theta)$ is included in $\{u<0\}$ is clearly satisfied.  

Recall that $G(t,y)$ has the property 
\be \label{c*}
\mathcal J(G, B_R^+) = C^* \, \log R + O(1),
\ee
for some constant $C^*>0$. Moreover, $G$ is a minimizer of $\mathcal J$ in $B^+_R$ among functions with values between $-1$ and $1$ which agree with $G$ on $\p B_R^+ \setminus \{y=0\}$.

Before we proceed with the proof of Proposition \ref{TH} we need some energy bounds for functions defined in 
half-squares
 $$Q_l:=[-l,l] \times [0,l].$$

\begin{lem}\label{l4.5}
a) Assume that $V$ is Lipschitz, defined in $Q_{l+1} \subset \R^2_+$, $|V| \le 1$ and 
\be\label{40.1}
\mbox{$V(t,0) \le -1+ \gamma^2$ if $t \le -\frac l 2$, and $V(t,0) \ge 1-\gamma^2$ if $t \ge \frac l2$} 
\ee
for some small $\gamma$.
Then for all sufficiently large $l$ we have 
\be\label{40}
\mathcal J(V,Q_l) \ge (C^* - \gamma)\log l.
\ee

b) Moreover, if we assume that there exist two points $s_1$, $s_2$ in $[-l/2.l/2]$ with $|s_1-s_2| \ge \theta_0$, such that
$$|V(t+s_i,y)-G(t,y)| \le c(\theta_0) \quad \mbox{in} \quad \mathcal B^+_{l^\sigma} ,$$
for some given $\sigma>0$ small and $c(\theta_0)$ sufficiently small, then
\be\label{41}
\mathcal J(V,Q_l) \ge (C^* + c_0(\sigma))\log l,
\ee
for some constant $c_0(\sigma)>0$.

\end{lem}

The proof of Lemma \ref{l4.5} is postponed till the end of this section.

\

{\it Proof of Proposition \ref{TH}}

First we remark that $ l \ge \theta_0 \eps_0^{-1} \to \infty$ as $\eps_0 \to 0$.

Let $A$ be the rescaling of the $0$ level set of $u$ given by
$$ (x',x_n) \in \{ u =0 \} \mapsto (z',z_n) \in A$$
$$ z= T x, \quad (z', z_n)= T(x',x_n):=(x' \, l^{-1}, x_n \theta^{-1}).$$
Our hypothesis is that $A \subset \mathcal C(1,1)$ and we want to show that in the cylinder $|z'| \le \frac 1 4$ the set $A$ is included 
either in $z_n \le 1-\omega_0$ or in $z_n \ge -1+\omega_0$. 

We view $A$ as a multivalued graph over $z' \in B_1'$. 

Let us assume that we touch $A$ by below at a point $z_0 \in A$ with the graph of a quadratic polynomial $P_{p'}^\mu$ 
of opening $- \mu$ and vertex $p' \in B_{1/3}'$
$$z_n= P^\mu_{p'}(z'):= - \frac{\mu}{2} |z'-p'|^2 + c_{p'},  \quad \quad \mbox{for some constant} \quad c_{p'} \le -1 + \frac \mu 8,$$
and $\mu \in [\omega_0,1]$ with $\omega_0$ a small universal constant to be specified later. 

We claim that Lemma \ref{l3} and Corollary \ref{c1} imply that $A$ satisfies the following two properties:

a) $A$ contains a graph which is fully included in the cylinder $z_0 + \mathcal C (l^{\sigma-1},2\mu).$

b) $A$ cannot be touched at $z_0$ in a $B_{r_0}(z_0)$ neighborhood with $$r_0:=l^{\sigma /2 -1},$$ by the graph 
$$\left \{z_n=P^\mu_{p'} + 4 \mu \Lambda ((z'-z_0') \cdot e')^2, \quad |z'-z_0'| \le r_0 \right \}, $$
with $e' \in \R^{n-1}$ a unit direction. Here $\sigma$ and $\Lambda$ represent the universal constants that appear in  Lemma \ref{l2} and Corollary \ref{c1}.

Indeed, the restrictions on $p'$ and $c_{p'}$ imply that $|z_0'| \le 5/6$ and $z_0 \cdot e_n \le -1 + \frac \mu 8$.
The corresponding point 
$$\mbox{$x_0 :=T^{-1} z_0$ then satisfies $|x_0'| \le 5l/6$, $x_0 \cdot e_n \le (-1 + \frac \mu 8) \theta$. }$$
Moreover, if the constant $C'$ in our hypothesis is chosen large depending on $\omega_0$, 
then the ball of radius $$q=l^2 (2 \mu \theta)^{-1} $$ which is tangent to $\{u=0\}$ at $x_0$ by below is globally included in $\{u<0\}$.
The outer normal $\nu$ to this ball at $x_0$ satisfies $|\nu- e_n| \le \mu \theta l^{-1}$ and Lemma \ref{l3} implies that
\be\label{43}
\{u=0\} \cap B_{q^\sigma}(x_0) \subset \{|(x-x_0) \cdot \nu| \le q^{-3/4}\}.
\ee
We use that $ q \ge c l \eps_0^{-1} \ge C l$ provided that $\eps_0$ is chosen small, and $$q^{-3/4}\le l^{-3/4} \le \omega_0 \theta_0 \le \mu \theta,$$ 
and property a) above follows by rescaling back \eqref{43} to the $z$ variable.

Property b) holds since otherwise, as above, we end up at the point $x_0$ 
with a surface as in Corollary \ref{c1} tangent to $\{u=0\}$ by below in a $\frac 12 l^{\sigma /2}$- neighborhood of $x_0$. 
This neighborhood includes $B_{q^{\sigma/6}}(x_0)$ since $l^3 \ge C l^2 \ge C q$ and we reach a contradiction and the claim is proved.

By Remark \ref{r3} we obtain in \eqref{43} also information on the whole profile $U$
$$|U-G((x-x_0)\cdot \nu,y)| \le q^{-3/4} \quad \mbox {in} \quad \mathcal B_{q^\sigma}^+(x_0).$$
This implies that for each $x \in B_{l^\sigma}(x_0) \cap \{u=0\}$ we have 
\be\label{44}
|U(x+t e_n ,y)-G(t,y)| \le C l^{-3/4} +C \theta l^{-1} \le \rho(\eps_0),
\ee
if $|(t,y)| \le l^\sigma$, and $\rho(\eps_0) \to 0 $ as $\eps_0 \to 0$. In the inequality above, we used 
$$|(x+te_n -x_0) \cdot \nu - t | \le |(x-x_0)\cdot \nu| + |\nu - e_n||t| \le q^{-3/4} + \theta l^{-1} |t|,$$
and \eqref{NG}.

Properties a) and b) above state that $A$ satisfies the hypotheses for the general version of Weak Harnack Inequality proved in \cite{DS}. 
Indeed, by property b) the set $A$ cannot be touched by below with the 
family of surfaces $\overline{\mathcal P} ^\mu_{8 \Lambda}(r_0)$, $\mu \in [\omega_0,1]$, 
in an neighborhood that contains at least a ball of radius $r_0$ around the contact point. 
On the other hand, by property a) Harnack inequality already holds in a 
$C^*(n,\Lambda) r_0 \le l^{\sigma-1}$ neighborhood of a contact point, 
with $C^*(n,\Lambda)$ the universal constant depending only on $n$ and $\Lambda$ 
which appears in Proposition 1.4 in \cite{DS}. 
Now we can apply the Proposition 1.4 of \cite{DS} and conclude that if
$$ A \cap \{|z'| \le 1/4, z_n \le -1+\omega_0\} \ne \emptyset$$ then
$A$ contains a graph $\underline{A} \subset A$ with
$$\underline A \subset B'_{1/4} \times [-1, -1+K \omega_0],$$
with $K=K(n,\Lambda)$ universal, such that
\be \label{45}
 \mathcal H^{n-1}\left(\pi_n(\underline{A}) \right) \ge (1- \frac 14) \mathcal H^{n-1}(B'_{1/4}),
\ee
where $\pi_n$ denotes the projection in the $z'$ variable. 

We choose $\omega_0$ small, depending on $K$ such that $K \omega_0 \le 1/2$ hence 
$$\underline{A} \subset \{z_n \le -\frac 1 2\}.$$
Similarly, if in the cylinder $z' \in B_{1/4}'$ the set $A$ intersects $z_n \ge 1-\omega_0$, 
then we can find a graph $\overline {A} \subset A \cap \{z_n \ge 1/2 \}$ which satisfies \eqref{45} as well.

We will reach a contradiction by estimating the energy $\mathcal J(U,A_{l/2})$ where
$$A_l:=\mathcal C(l,l) \times [0,l] \subset \R^{n+1}.$$

Notice that in $\mathcal C(\frac l 2,\frac l 2)$ the function $U(x,0)$ is sufficiently close to $\pm 1$ 
away from a thin strip around $x_n=0$. 
Indeed, we can use barrier functions as in Lemma \ref{l2} (see \eqref{22}) and bound $U$ by above and below in terms 
of the functions $G_{l/2}(x_n \pm \theta,y)$. This implies that, for any constant $\gamma$ small, we have 
\be\label{46}
|U(x,0)| \ge 1- \gamma^2 \quad \mbox{in $\mathcal C( \frac l 2, \frac l2)$ if $|x_n| \ge C(\gamma)+\theta$.}
\ee

For each $x' \in B'_{l/2}$ we denote by $Q_l(x')$ the 2D half square of size $l$ in the $(x_n,y)$-variables centered at $(x',0) \in \R^n$ as 
$$ Q_{l}(x'):=\{  (x',t,y)|  \quad |t| \le l, \quad y \in [0,l]\}.$$
Now we can apply Lemma \ref{l4.5} part a) and obtain
$$\mathcal J(U(x', \cdot),Q_{l/2}(x')) \ge (C^*-\gamma) \log l.$$
On the other hand, if $$x' l^{-1} =z' \in  \pi_n(\overline{A}) \cap \pi_n(\underline {A}),$$ then by \eqref{44}, we satisfy the hypotheses of part b) of Lemma \ref{l4.5} and obtain
$$\mathcal J(U(x', \cdot),Q_{l/2}(x')) \ge (C^* + c_0(\sigma)) \log l.$$

In conclusion, after integrating in $x' \in B_{l/2}'$ the inequalities above, and using \eqref{45} for $\overline{A}$, $\underline{A}$, we obtain that
$$\mathcal J(U,A_{l/2}) \ge (C^* + c_1)\log l \, \,  \mathcal H^{n-1}(B`_{l/2}),$$
for some $c_1>0$ universal, provided that we choose the constant $\gamma$ sufficiently small. 

This contradicts Lemma \ref{l1conv} below if $\eps_0$ is sufficiently small.

\qed

\begin{lem}\label{l1conv}
\begin{equation}{\label{conv2}}
\mathcal J( U , A_{l/2}) \le C^* \log l \left(  \mathcal H^{n-1}(B_{l/2}') + \eta(\eps_0)\,  l^{n-1} \right) .
\end{equation}
with $\eta(\eps_0) \to 0$ as $\eps_0 \to 0$.
\end{lem}

\begin{proof}
We interpolate between $U$ and $V(x,y):=G(x_n,y)$ as
$$H=(1-\varphi)U + \varphi V.$$
Here $\varphi $ is a cutoff Lipschitz function such that
$\varphi =0$ outside $A_{l/2}$, $\varphi=1$ in $\mathcal R$ and $|\nabla  \varphi| \le 8/(1+y)$ in $A_{l/2}\setminus \mathcal R$,  where $\mathcal R$ is the cone
$$\mathcal R:= \{(x,y) | \quad \max \{|x'|,|x_n|\} \le l/2 -1 - 2y  \}. $$
By minimality of $U$ we have
$$\mathcal J(U,A_{l/2}) \le \mathcal J(H,A_{l/2})= \mathcal J(V, \mathcal R) + \mathcal J(H, A_{l/2} \setminus \mathcal R).$$
By \eqref{c*}, $$\mathcal J(V, \mathcal R) \le \mathcal J(V, A_{l/2}) \le (C^*\log l +O(1)) \mathcal H^{n-1}(B_{l/2}'),$$
and we need to show that 
\be\label{JH}
\mathcal J(H, A_{l/2} \setminus \mathcal R) \le \eta \, l^{n-1} \log l
\ee with $\eta$ arbitrarily small.
We have
$$\mathcal J(H, A_{l/2} \setminus \mathcal R) \le 4 \int_{A_{l/2} \setminus \mathcal R}  |\nabla \varphi|^2(V-U)^2 + |\nabla (V-U)|^2   \, dx dy $$
\be\label{JHA}
+ \int_{D} W(u) + W(v) + C(v-u)^2 dx.
\ee
with $D:=\mathcal C(\frac l2, \frac l2) \setminus \mathcal C( \frac l 2-1, \frac l2-1)$. The second integral is bounded by $C l^{n-1}$.

Next we bound the first integral. As in \eqref{46}, $u$ and $v$ are sufficiently close to $\pm 1$ in $\mathcal C( \frac l2, \frac l 2)$ away from a thin strip around $x_n=0$,
$$|v-u| \le  \gamma^2 \quad \mbox{in $\mathcal C(l/2,l/2)$ if $|x_n| \ge C(\gamma)+\theta$.}$$
with $C(\gamma)$ large, depending on the universal constants and $\gamma$. 
Then in the region
$$S:= \{1 \le y  \le \gamma^2( |x_n|- C(\gamma)-\theta)  \}, $$
the extensions $U$ and $V$ satisfy
\be\label{40.0}
|V-U| \le C \gamma^2, \quad \quad \quad |\nabla (V-U)| \le C \gamma^2 y^{-1}.\ee
At all other points we use that $|U|,|V| \le 1$, $|\nabla U|, |\nabla V| \le C/(1+y)$ and we see from \eqref{JHA} that
\begin{align*}
\mathcal J(H, A_{l/2} \setminus \mathcal R) \le & C l^{n-1} + C \int \gamma^2 y^{-2} \chi_{(A_{l/2} \setminus \mathcal R) \cap S} + (1+y)^{-2}\chi_{(A_{l/2} \setminus \mathcal R) \setminus S} \, \, dxdy\\
 \le & C(\gamma) l^{n-1} + C \gamma^2 l^{n-1} \log l \le \eta \, \,  l^{n-1} \log l.
\end{align*}
for all $l$ large, provided that $\gamma$ is chosen small, and \eqref{JH} is proved.

\end{proof}

We conclude this section with the proof of the Lemma \ref{l4.5}.

\

{\it Proof of Lemma \ref{l4.5}}

The proof of \eqref{40} follows by the same argument of Lemma \ref{l1conv} above restricted to the case $n=1$ (now we denote $x_n$ by $t$). First we may assume that $V$ is minimizing the energy among functions which have prescribed boundary data on $\p Q_{l+1} \setminus \{y=0\}$ and are constrained to \eqref{40.1} on $y=0$. We interpolate between $V$ and $G$ 
$$H=(1-\varphi)G + \varphi V,$$
with $\varphi$ defined as above and obtain
$$\mathcal J(H, Q_l) = \mathcal J(V, \mathcal R) + \mathcal J(H, Q_l \setminus \mathcal R).$$
As in \eqref{JHA} we can use that in the region 
$$S:=\{1 \le y \le \gamma^2(|t|-l/2)\}$$
the functions $V$ and $G$ satisfy the estimate \eqref{40.0} and obtain
$$\mathcal J(H, Q_l)  \le C(\gamma) + C \gamma^2 \log l \le \frac \gamma 2 \log l.$$  
Thus,
$$\mathcal J(V, \mathcal R) \ge \mathcal J(H, Q_l) -\frac \gamma 2 \log l \ge \mathcal J(G, Q_l) -\frac \gamma 2 \log l,$$
and \eqref{40} follows by \eqref{c*}. Above we used that $H=G$ on $\p Q_l \setminus \{y=0\}$ and the fact that $G$ is a minimizer of $\mathcal J$ in $Q_l$.

For the second part we use $\bar V$, the monotone increasing rearrangement in the $t$ direction of $V$.
Denote by $$\Gamma(D):=\{ z=V(t,y)| \quad (t,y) \in D \} \subset \R^3$$ the graph of $V$ over the set $D$, and 
let $T$ be the angular region
$$T:= \{ y \ge |t-x_1| \} \cap B_{l^\sigma/2}.$$
Notice that our hypotheses imply that $|s_1-s_2| \ge l^{\sigma}$ provided that $c(\theta_0)$ is sufficiently small. This means that
that the projection of $\Gamma(T) $ along the $t$ direction is included in the projection of $\Gamma(Q_l \setminus T)$. 

From the theory of monotone increasing rearrangements (see \cite{K}) we obtain that
 $$\mathcal J(V,Q_l) \ge \mathcal J(\bar V, Q_l) + \int_T V_t^2 \, dtdy.$$
 On each horizontal segment $\ell _y$ of $T$ at height $y \in [1,l^\sigma/4]$, we use that $V(t,y)$ and $G(t-x_1,y)$ are sufficiently close and obtain
 $$\int_{\ell_y} V_t^2 dt \ge \frac c y,$$ hence
$$\int_T V_t^2 \, dtdy \ge 2 c_0(\sigma) \log l.$$ 
 Notice that the rearrangement $\bar V$ still satisfies the hypothesis in part a), and then the conclusion follows from a).   
 
\qed

\section{Improvement of flatness}

We state the improvement of flatness property of minimizers.

\begin{thm}[Improvement of flatness]{\label{c1alpha}}

Let $U$ be a minimizer of $\mathcal J$ and assume
\be\label{50}
0 \in \{u=0\} \cap \mathcal C(l,l) \subset \mathcal C(l,\theta),
\ee
and 
\be\label{51}
\mbox{the balls of radius $C' l^2 \theta^{-1}$ ($C'$ universal)} 
\ee
which are tangent to 
$\mathcal C(l,\theta)$ by below and above at $\pm \theta e_n$ are included in $\{u<0\}$ respectively $\{u>0\}$.

There exists universal integers $m_0 \ge 1$, $m_1 \ge 0$ such that if
the balls of radius 
$$R_m:= l_m^2 \theta_m^{-1}, \quad \quad l_m:=2^m l, \quad \theta_m:=2^{\frac 32 m} \theta, \quad \quad m:=0,1,\cdots, m_1,$$
tangent to $\mathcal C(l_m,\theta_m)$ by below and above at $\pm \theta_m e_n$ are included in $\{u<0\}$ respectively $\{u>0\}$,
and if
$$\theta l^{-1} =:  \eps \le \eps_1(\theta_0), \quad \theta \ge \theta_0,$$
with $\eps_1(\theta_0)$ sufficiently small, then $\eqref{50}$, $\eqref{51}$ hold for $\bar l$, $\bar \theta$ after a rotation with
$$\{u=0 \} \cap   \mathcal C_\xi (\bar l, \bar l) \subset C_\xi(\bar l, \bar \theta), \quad \quad \bar l:=2^{-m_0} l, \quad \bar \theta:=2^{-\frac 32 m_0} \theta.$$
\end{thm}

Here $\xi \in \R^n$ is a unit vector and $C_\xi(\bar l, \bar \theta)$ represents the cylinder with axis $\xi$, base $\bar l$ and height $\bar \theta$.

As a consequence of this flatness theorem we obtain our main theorem.

\begin{thm}{\label{planelike}}
Let $U$ be a global minimizer of $\mathcal J$. Suppose that the $0$ level set $\{u=0\}$ is asymptotically flat at $\infty$.
Then the $0$ level set is a hyperplane and $u$ is one-dimensional.
\end{thm}

\begin{proof} Without loss of generality assume $u(0)=0$. Fix $\theta_0>0$, and $\eps \ll \eps_1(\theta_0)$. From the hypotheses we can find $l$, $\theta$ large such that $\theta l^{-1} = \eps$ and, after eventually a rotation, conditions \eqref{50}, \eqref{51} hold for all $l_m$, $\theta_m$
$$l_m:=2^m l, \quad \theta_m:=2^{\frac 32 m} \theta, \quad \quad \mbox{with} \quad m \in \{m_1,m_1-1,\cdots , 1- m_0\}.$$
Then, by Theorem \ref{c1alpha}, \eqref{50}, \eqref{51} hold also for $m=-m_0$ after a rotation. 
It is easy to check that we can apply Theorem
\ref{c1alpha} repeatedly till the
height of the cylinder becomes less than $\theta_0$. We conclude that $\{u=0\}$ is trapped in a cylinder with flatness less than
$\eps$ and height between $2^{-\frac 32 } \theta_0$ and $\theta_0$. We let first $\eps \to 0$ and then $\theta_0 \to 0$ and obtain the desired conclusion.

\end{proof}

\

{\it Proof of Theorem \ref{c1alpha}}

The proof is by compactness and it follows from Proposition \ref{TH} and Proposition \ref{p1}.
Assume by contradiction that there exist $U_k$, $\theta_k$, $l_k$ such that 

a) $U_k$ is a minimizer of $\mathcal J$, and satisfies \eqref{50}, \eqref{51} for $l_k$, $\theta_k$, together with the second hypothesis for $l_{k,m}$, $\theta_{k,m}$ and $m \in \{0,1,\cdots,m_{1,k}\}$,

b)  $\theta_k \ge \theta_0$, $\theta_k l_k^{-1}= \eps_k \to 0$, $m_{1,k } \to \infty$, as $k \to \infty$,

c) the conclusion of Theorem \ref{c1alpha} does not hold for $u_k$ with a constant $m_0$ depending only on $n$ and $C'$ which we will specify later. 

\

Let $A_k$ be the rescaling of the $0$ level sets given by
$$ (x',x_n) \in \{ u_k =0 \} \mapsto (z',z_n) \in A_k$$
$$ z'=x'l_k^{-1}, \quad z_n=x_n \theta_k^{-1}.$$

{\it Claim 1:} $A_k$ has a subsequence that converges uniformly on $|z'|
\le
1/2$ to a set $A_{\infty}=\{(z',w(z')), \quad |z'| \le 1/2 \}$ where $w$
is a Holder continuous function. \\

{\it Proof:} Fix $z_0'$, $|z_0'| \le 1/2$. We apply Proposition \ref{TH} for the function $u_k$ in the
cylinder of base $B'_{l/2}(l_kz_0')$ and height $2 \theta_k$
in which the set $\{ u_k =0 \}$ is trapped. Thus, there exist an increasing function $\varepsilon_0(\theta)>0$,
$\varepsilon_0(\theta) \to 0$ as $\theta \to 0$, such that $\{ u_k =0 \}$
is
trapped in the cylinder of base $B'_{l/2}(l_kz_0')$ and height $2(1-\frac{\omega_0}{2})\theta_k$
provided that $4\theta_k l_k^{-1} \le \varepsilon_0(2\theta_k)$. Rescaling back we find that the oscillation of $A_k$ in the $z_n$ variable in $B'_{1/8}(z_0')$ is bounded by $2(1-\frac{\omega_0}{2})$.
It is not difficult to see that we can apply the Harnack inequality 
repeatedly and we find that the oscillation of $A_k$ in the $z_n$ variable in $B'_{2^{-2m -1}}(z_0')$ is bounded by $2(1-\frac{\omega_0}{2})^m$
provided that
$$ \eps_k \le 4^{-m-1} \varepsilon_0 \left (2(1-\frac{\omega_0}{2})^m
\theta_k \right).$$
Since these inequalities are satisfied for all $k$ large, the claim follows from a version of Arzela-Ascoli Theorem.\\

{\it Claim 2:} The function $w$ is harmonic (in the viscosity sense).\\

{\it Proof:} The proof is by contradiction. Fix a quadratic polynomial
$$z_n=P(z')=\frac{1}{2}{z'}^TMz' + \xi \cdot z', \quad \quad \|M \|
< \delta^{-1}, \quad 2|\xi| <\delta^{-1}$$
such that $ tr \, M >\delta $, $P(z')+ \delta  |z'|^2$ touches
the graph of $w$, say, at $0$ for simplicity, and stays below $w$ in
$|z'|<\delta$, for some small $\delta$.

Thus,
for all $k$ large we find points $({z_k}',{z_k}_n)$ close to $0$ such that
$P(z')+ const$ touches $A_k$ by below at $({z_k}',{z_k}_n)$ and stays
below it in $|z'-{z_k}'|< \delta/2$.

This implies that, after eventually a translation, there exists a surface
$$\Gamma:=\left \{ x_n =
\frac{\theta_k}{l_k^2} \sum \frac{a_i}{2} x_i^2+ \frac{\theta_k}{l_k} \xi_k
\cdot x'
\right \}, \quad |\xi_k|\le  \delta^{-1}, \quad |a_i| \le \delta^{-1},$$
with $$\sum a_i > \delta$$ that touches $\{ u_k = 0 \}$ at the origin and stays below it in the cylinder $\mathcal C(\frac 12 \delta l_k, \theta_k)$.

Now we apply Proposition \ref{p1} with
$$\bar l:= \delta^3 l_k, \quad \bar \theta:=\theta_k, \quad \bar R:= {\bar l}^2 \bar \theta^{-1}, \quad \bar \delta:= \delta^2, \quad  \bar a_i:= \frac{\theta_k}{l_k^2} \, a_i = \frac{\bar \delta^3}{\bar R}a_i  , \quad \bar b'= \frac{\theta_k}{l_k}  \, \xi_k.$$
The hypotheses are satisfied since 
$$\bar l \le c(\delta)\eps_k \bar R \le \bar \delta^3 \bar R,  \quad \quad \bar l \ge c(\theta_0)\bar R ^{1/2} \ge \bar R^{1/3},$$
and
$$|\bar a_i|, |b'|\bar l^{-1} \le \delta^2 {\bar R}^{-1} = \bar \delta {\bar R}^{-1}.$$
Moreover, by property a) above the balls of radius $$\bar R_m={\bar l}^2_m {\bar \theta}_m^{-1}, 
\quad \mbox{with} \quad {\bar l}_m:=2^m \bar l, \quad {\bar \theta}_m:=2^{\frac 32 m} \theta, \quad m \in \{0,1,\cdots, m_{k,1}\},$$
and tangent to $\mathcal C (\bar l _m, \bar \theta_m)$ are included in $\{u<0\}$ and respectively $\{u>0\}$.
By Proposition \ref{p1} we conclude that 
$$\sum \bar a_i  \le \bar \delta^4 \bar R^{-1} \quad \Longrightarrow \quad  \sum a_i \le \bar \delta= \delta^2,$$
and we reached a contradiction, and Claim 2 is proved.

\

Since $w$ is harmonic, and $w(0)=0$,
$$|w-\xi' \cdot z'| \le C(n) |z'|^2 , \quad \quad |\xi'|\le C(n),$$
with $C(n)$ a constant depending only on $n$. We choose $2^{-m_0}=\eta$ sufficiently small such that
$$C(n) \eta^2 \le \frac 12 \eta^{3/2}, \quad \quad 4 C(n) C' \eta^\frac 12 \le 1.$$ 
Now it is easy to check that after rescaling back, and using the fact that $A_k$ converge uniformly to the
graph of $w$, the sets $\{u_k=0\}$ satisfy the
conclusion of the Theorem \ref{c1alpha} for all $k$ large enough, and we reached a contradiction.

\qed


\begin{thebibliography}{9999}

\bibitem[CC]{CC} Cabre X., Cinti E., Sharp energy estimates for nonlinear fractional diffusion equations. {\it Calc. Var. Partial Differential Equations 49} (2014), no. 1-2, 233--269.


\bibitem[D]{D} De Giorgi E., Convergence problems for functional and
operators. {\it Proc. Int. Meeting on Recent Methods in Nonlinear
Analysis}
(Rome, 1978), 131--188.

\bibitem[DS]{DS} De Silva D., Savin O., Quasi-Harnack inequality, {\it Preprint}

\bibitem[DKW]{DKW} Del Pino M., Kowalczyk M., Wei J., On De Giorgi Conjecture in dimension $N \ge 9$. {\it Ann. of Math.} (2). 174 (2011), no. 3, 1485--1569.

\bibitem[DSV]{DSV} Dipierro S., Serra J., Valdinoci E., Improvement of flatness for nonlocal phase transitions.  {\it Preprint}  arXiv:1611.10105.

\bibitem[FS]{FS} Figalli A., Serra J. On stable solutions for boundary reactions: a De Giorgi-type result in dimension 4+1, {\it Preprint}    arXiv:1705.02781.

\bibitem[K]{K} Kawohl B., Rearrangements and convexity of level sets in PDE. {\it Lecture Notes in Mathematics, 1150.} Springer-Verlag, Berlin, 1985. iv+136 pp.

\bibitem[PSV]{PSV}  Palatucci G., Savin O., Valdinoci E., Local and global minimizers for a variational energy involving a
fractional norm. {\it Ann. Mat. Pura Appl. (4) 192} (2013), no. 4, 673--718.

\bibitem[M]{M} Modica L., $\Gamma $-convergence to minimal surfaces
problem and global solutions of $\Delta u=2(u\sp{3}-u)$. Proceedings of
the International Meeting on Recent Methods in Nonlinear Analysis (Rome,
1978), pp. 223--244, Pitagora, Bologna, 1979.

\bibitem[S1]{S1} Savin O., Regularity of flat level sets in phase transitions. {\it Ann. of Math.} (2) {\bf 169} (2009), no.1, 41--78.

\bibitem[S2]{S2}  Savin O., Some remarks on the classification of global solutions with asymptotically flat level sets. {\it Calc. Var. Partial Differential Equations 56} (2017), no. 5, Art. 141, 21 pp.



\bibitem[S3]{S3} Savin O., Rigidity of minimizers in nonlocal phase transitions. {\it Analysis and PDE} (to appear) arXiv:1610.09295.



\bibitem[SV]{SV} Savin, O., Valdinoci E.,  $\Gamma$-convergence for nonlocal phase transitions. {\it Ann. Inst. H. Poincare Anal. Non Linéaire 29} (2012), no. 4, 479--500.



\end{thebibliography}
\end{document}